\documentclass[a4paper,10pt]{article}

\usepackage[utf8]{inputenc}
\usepackage[T1]{fontenc}
\usepackage{xspace,
	amsthm,
	amsmath,
	amssymb,
	bbm,
	tikz,
	graphicx,
	subfig,
	keyval}
\DeclareGraphicsExtensions{.png, .pdf}

\usepackage[hidelinks]{hyperref}

\usepackage{geometry}
\tikzstyle{nodino}=[circle,draw,fill,inner sep=0pt,minimum size=0.5mm]
\tikzstyle{infinito}=[circle,inner sep=0pt,minimum size=0mm]
\tikzstyle{nodo_vuoto}=[circle,draw,inner sep=0pt, minimum size=0.5*width("k")]
\usetikzlibrary{graphs}
\tikzset{every loop/.style={min distance=10mm,in=300,out=240,looseness=10}}
\tikzset{place/.style={circle,thick,draw=blue!75,fill=blue!20,minimum
		size=6mm}}
\tikzset{place2/.style={circle,thick,draw=red!75,fill=red!20,minimum
		size=6mm}}
\tikzstyle{nodo}=[circle,draw,fill,inner sep=0pt,minimum size=0.5* width("k")]

\newcommand{\rr}{{\mathbb R}}
\newcommand{\R}{{\mathbb R}}
\newcommand{\cc}{{\mathbb C}}

\newcommand{\nn}{{\mathbb N}}
\newcommand{\NN}{{\cal N}}
\newcommand{\G}{{\mathcal{G}}}
\newcommand{\udot}{\|u'\|_2}
\newcommand{\uLp}{\|u\|_p}
\newcommand{\uLtwo}{\|u\|_2}
\newcommand{\HmuG}{H_\mu^1(\mathcal{G})}
\newcommand{\uLsix}{\|u\|_6}

\newcommand\vv{\textsc{v}}
\newcommand{\eps}{\varepsilon}
\newcommand{\cost}{\kappa_\mu}

\newcommand{\mup}{\mu_{\R^+}}
\newcommand{\mur}{\mu_\R}
\newcommand{\mug}{\mu_\G}
\newcommand{\dx}{\,dx}
\newcommand{\EE}{E}

\theoremstyle{plain} 
\newtheorem{thm}{Theorem}[section] 
\newtheorem{cor}[thm]{Corollary} 
 
\newtheorem{proposition}[thm]{Proposition} 

\theoremstyle{definition}

\theoremstyle{definition}

\theoremstyle{remark} 
\newtheorem{rem}{Remark}[section]

\date{}

\title{Variational and stability properties of constant solutions \\ to the NLS equation on compact metric graphs}

\author{Claudio Cacciapuoti$^\dagger$, Simone Dovetta$^{\sharp}$, Enrico Serra$^\sharp$
	\\ \ \\ \ \\
	{\small  $^\dagger$Dipartimento di Scienza ed Alta Tecnologia} \\
	{\small Universit\`a degli Studi dell'Insubria} \\
	{\small Via Valleggio 11, 22100 Como, Italy} \\ \ \\
	{\small $^\sharp$Dipartimento di Scienze Matematiche ``G.L. Lagrange''}\\
	{\small Politecnico di Torino }\\
	{\small Corso Duca degli Abruzzi 24, 10129 Torino, Italy} \\ \ \\
	}

\begin{document}
	
	\maketitle
	
	\begin{abstract} We consider the nonlinear Schr\"odinger equation with pure power nonlinearity on a general compact metric graph, and in particular its stationary solutions with fixed mass. Since the the graph is compact, for every value of the mass there is a constant solution. Our scope is to analyze (in dependence of the mass) the variational properties of this solution, as a critical point of the energy functional:  local and global minimality, and  (orbital) stability. We consider both the subcritical regime and the critical one, in which the features of the graph become relevant. We describe how the above properties change according to the topology and the metric properties of the graph.
	\end{abstract}

\smallskip



\section{Introduction}

Partial differential equations on one-dimensional networks (metric graphs) arise naturally in modeling almost one-dimensional ramified structures. Even though a rigorous justification of the one-dimensional approximation  is, in general, not easily achieved, these models are extensively studied in view of their applications in many different fields,  such as  neurobiology, chemical physics, engineering, and  physics. 

A first application of PDEs on  metric graphs, in chemical physics, to study the spectrum of the naphthalene molecule, dates back to 1953 with the work of Ruedenberg and Scherr \cite{RS53}. Yet rigorous mathematical analysis on these models has been relatively quiescent until the last two decades, which have seen a revived interest in the subject, as testified for example by the conference proceedings \cite{alimehmeti-proc, berkolaiko-etal, exner-etal,  mugnolo-proc}. Nowadays several introductory books and review articles on the mathematical aspects of PDEs on  metric graphs  are available, and we refer e.g. to  \cite{BK, kuchment1, kuchment2, kuchment3, mugnolo}. 

In view of the applications it is natural  to consider the mathematical aspects of   {\em nonlinear} equations on metric graphs. The number of works on this topic is huge and constantly growing, and we will not make an attempt  to fully cover the literature. We shall instead focus our attention on the main subject of our analysis which is the nonlinear Schr\"odinger (NLS) equation with pure power nonlinearity. This  equation,  for a suitable choice of  the power of the nonlinearity,  can for instance be used to model Bose-Einstein condensates in  Josephson junctions, see e.g. \cite{lorenzo-etal}. \\ 

A (connected) metric graph $\mathcal G$ (see Fig. \ref{figuno})  is formed by line segments or half-lines, called \emph{edges}, with some of their endpoints glued together at points called \emph{vertices}. We denote the set of all the edges by $\EE$ and the set of all the vertices by $V$.
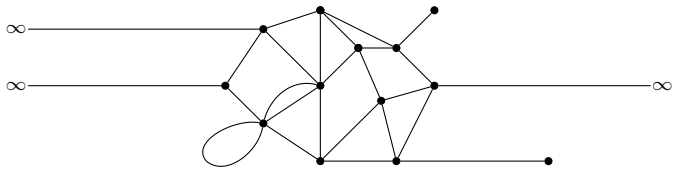
\begin{figure}
\begin{center}
\begin{tikzpicture}[xscale= 0.5,yscale=0.5]
\node at (-.5,2) [nodo] (02) {};
\node at (2,2) [nodo] (22) {};
\node at (2,4) [nodo] (24) {};
\node at (3.6,1.6) [nodo] (42) {};
\node at (3,3) [nodo] (33) {};
\node at (5,2) [nodo] (52) {};
\node at (3,3) [nodo] (32) {};
\node at (4,3) [nodo] (43) {};
\node at (.5,3.5) [nodo] (04) {};
\node at (2,4) [nodo] (24) {};
\node at (.5,1) [nodo] (11) {};
\node at (2,0) [nodo] (20) {};
\node at (4,0) [nodo] (40) {};
\node at (8,0) [nodo] (80) {}; 
\node at (5,4) [nodo] (54) {};
\node at (-6,2) [infinito] (meno) {$\scriptstyle\infty$};
\node at (-6,3.5) [infinito] (menoalt) {$\scriptstyle\infty$};
\node at (11,2) [infinito] (piu) {$\scriptstyle\infty$};

\draw[-] (40)--(80);
\draw[-] (02)--(04);
\draw[-] (04)--(24);
\draw[-] (04)--(22);
\draw[-] (11) to [out=75,in=165] (22);
\draw[-] (24)--(22);
\draw[-] (02)--(11);
\draw[-] (11)--(22);
\draw[-] (11)--(20);
\draw[-] (11) to [out=170,in=130] (-1,0);
\draw[-] (-1,0) to [out=-40,in=260] (11);
\draw[-] (20)--(22);
\draw[-] (22)--(33);
\draw[-] (24)--(33);
\draw[-] (24)--(43);
\draw[-] (33)--(43);
\draw[-] (43)--(52);
\draw[-] (33)--(42);
\draw[-] (20)--(42);
\draw[-] (20)--(40);
\draw[-] (40)--(42);
\draw[-] (42)--(52);
\draw[-] (40)--(52);
\draw[-] (43)--(54);
\draw[-] (02)--(meno);
\draw[-] (52)--(piu);
\draw[-] (04)--(menoalt);
\end{tikzpicture}
\end{center}
\caption{a connected metric graph}
\label{figuno}
\end{figure}

Choosing on it a coordinate $x_e$, each edge $e$ is identified either with an interval $I_e=[0,\ell_e]$, $\ell_e<+\infty$, or with $[0,+\infty)$ in the case of half-lines. 
With the shortest path distance, the set $\G$ becomes a metric space.

A metric graph is {\em compact}, as a metric space, if it has a  finite number of edges and all the edges have finite length (that is, there are no half-lines). In this paper we only consider compact metric graphs.

A function $u$ on $\G$ is a map $u:\G\to \mathbb C$. We denote by $u_e:I_e\to \cc$ the restriction of the function $u$ to the edge $e$, and by  $u'$ and $u''$ the functions with restriction to the edges given by $u_e'$ and $u_e''$. 

Various functional spaces can naturally be defined on $\G$. Indeed, the Lebesgue measure on $\G$ is introduced by using the one-dimensional Lebesgue measure on each interval $I_e$, for every edge $e$, and we let 
\[
\int_\G u \,dx := \sum_{e\in\EE}\int_{I_e} u_e \,dx_e 
\]
denote the integral of a function on $\G$. Then, for all $p\geq 1$, $L^p(\G)$ spaces are naturally defined as the sets 
\[
L^p(\G) := \{u:\G\to\cc\,:\, u_e\in L^p(I_e) \;\forall e\in\EE\}
\]
with corresponding norm $\|u\|_p := \left(\int_\G |u|^p \,dx\right)^{1/p}$. Accordingly, the space $L^2(\G)$ is also endowed with the scalar product $(u,v)_2:=   \int_\G  u \bar v \,dx$. 

Letting  $C(\G)$ be the space of continuous functions on $\G$, the space $H^1(\G)$ is then defined as 
\[
H^1(\G) := \{u:\G\to \cc\,:\, u\in C(\G), \; u_e \in H^1(I_e) \;\forall e\in\EE \}
\]
and it is endowed with the norm 
\[
\|u\|_{H^1(\G)} = \|u'\|_2 +\|u\|_2. 
\] 
Note that the continuity required in the definition of $H^1(\G)$ is a no-jump condition at the vertices.

On the other hand, due to the presence of the vertices,  a natural notion of continuity for the derivative $u'$ is not available. Hence,  also a natural definition of higher order Sobolev spaces is missing. In particular, to define the space $H^2(\G)$ it is necessary to assign a boundary condition on the  values of the  derivatives at the vertices of the graph. We set 
\[
H_K^2(\G) := \Big\{u\in H^1(\G)\;:\; u_e \in H^2(I_e) \; \forall e\in E,\quad \sum_{e\succ \vv} u_e'(\vv)=0\;\; \forall \vv \in V\Big\}, 
\]
where $e\succ \vv$ denotes the set of all the edges which have at least one endpoint coinciding with the vertex $\vv$, and  $u_e'(\vv)$ denotes the derivative in the outgoing direction.

A function $u\in H_K^2(\G)$ is said to satisfy   \emph{Kirchhoff} (also called \emph{free} or \emph{standard})  conditions at the vertices. The operator $-\Delta$ on $L^2(\G) $ defined by  $\textrm{dom}(-\Delta) := H_K^2(\G)$,  $-\Delta := - (\,\cdot\,)''$ is self-adjoint in $L^2(\G)$ and is usually called the  \emph{Kirchhoff} (or  \emph{free} or \emph{standard}) Laplacian. 

There are other choices of boundary conditions at the vertices that define self-adjoint realizations of the Laplacian on a metric graph (in fact also the continuity at the vertices is not a necessary condition). A complete characterization of such self-adjoint boundary  conditions is, for example, in  \cite{BK}. \\

Given a connected metric graph $\G$, the NLS equation we are interested in is 
\begin{equation}
\label{nlse}
i \frac{\partial}{\partial t} u = -  u'' -|u|^{p-2} u,
\end{equation}
whose associated energy functional $E:H^1(\G)\to\rr$ is 
\begin{equation}
\label{EQ-def energy INTRO}
E(u)=\frac{1}{2}\udot^2-\frac{1}{p}\uLp^p=\frac{1}{2}\int_{\G}|u'|^2\,dx-\frac{1}{p}\int_{\G}|u|^p\,dx\,.
\end{equation}
Here we consider both $p\in(2,6)$, the so-called \emph{$L^2$-subcritical regime}, and $p=6$, the \emph{critical regime}.

Of particular interest is the study of the critical points of the energy functional under the  constraint
\begin{equation}
\label{EQ-def mass constraint}
\uLtwo^2=\mu
\end{equation}
on the $L^2$-norm (\emph{mass}), which is a conserved quantity under the NLS flow, as well as the energy itself. Indeed, critical points of \eqref{EQ-def energy INTRO} satisfying \eqref{EQ-def mass constraint} are solutions of the (time-independent) equation 
\begin{equation}\label{stateq}
u''+|u|^{p-2}u-\lambda u = 0, \qquad u\in H_K^2(\G), 
\end{equation}
for some $\lambda \in \rr$ that arises as a Lagrange multiplier due to the presence of the  constraint \eqref{EQ-def mass constraint}.

With every solution $u$ of \eqref{EQ-def mass constraint}-\eqref{stateq} is then associated the time dependent solution of \eqref{nlse} given by $v(t) = e^{i\lambda t}u$. Since these solutions evolve in time simply through a multiplication by a phase factor, solutions of \eqref{stateq} are called \emph{stationary} (or \emph{bound})  \emph{states}. Particularly, we call \emph{ground states} those stationary solutions that globally minimize the energy \eqref{EQ-def energy INTRO} among functions fulfilling \eqref{EQ-def mass constraint}.

The study of the stationary solutions of an NLS equation has an interest in its own, in view of the fact that these  states are associated to physically observable ones. In general one is interested in their existence, multiplicity and stability/instability properties. 
In this context by stability we mean \emph{orbital stability}: a stationary solution $u$ is orbitally stable  if for all $\eps>0$ there exists $\delta >0$ such that if an initial datum $v_0 \in H^1(\G)$ for the NLS equation \eqref{nlse} satisfies $\inf_{\theta \in\rr}\|v_0 -e^{i\theta} u\|_{H^1}<\delta$, then the solution $v$ of \eqref{nlse} with initial condition $v_0$ satisfies $\sup_{t}\inf_{\theta\in\rr }\|v(t)-e^{i\theta} u\|_{H^1}<\eps$.  In other words, if the initial datum is close to the \emph{orbit} $e^{i\theta} u$, $\theta \in \rr$, then the solution to the NLS equation with that datum stays close to the orbit. \\

The analysis of the problem of the existence of stationary solutions for the NLS equation on metric graphs has been widely investigated on non-compact metric graphs first, with a prominent focus on  ground states. The existence of mass-constrained ground states was initiated in the series of works \cite{ACFN-jpa12, ACFN-epl12, ACFN-jde14, ACFN-aihp14}, for the special case of a star-graph (a graph obtained by merging $N$ half-lines). For generic non-compact metric graphs the problem was addressed in \cite{AST2015, AST-jfa16} in the subcritical regime and in \cite{AST-cmp17} in the critical one (see also \cite{li-li-shi}).  We also note the paper \cite{CFN-pre15}, where the stationary solutions for the NLS equation on the \emph{tadpole graph} (a loop with one half-line attached to it) were characterized  and  where  the existence of the ground state was only conjectured (the conjecture was then confirmed in  \cite{AST-jfa16}).  In the same setting, the stability/instability properties of the stationary states were analyzed in \cite{NPS-non15}.  

Related works on the existence of ground states on non-compact graphs are: \cite{DT-rxv18, ST-jde16, ST-na16, tentarelli-jmaa16} in a setting in which the nonlinearity is supported only on a compact part of the graph;  \cite{ADST-rxv18} considering a doubly periodic graph. We also mention \cite{astbound}, where multiple {\em bound} states are found as local, non necessarily global, minimizers of the energy functional.

For further discussions and additional references we refer to the review articles \cite{AST-ln, noja-rev}.\\

\section{Main results}

In this paper we only deal with compact graphs.  As far as ground states are concerned, in the subcritical regime $p \in (2,6)$ their existence for every value of the mass $\mu$ is immediately granted by the compactness of the graph, which entails the compactness of the embedding of $H^1(\G)$ into  $L^q(\G)$ for every $q$, see \cite{dovetta}. In the critical regime $p=6$ instead, the existence of a ground state is not a trivial matter. 

Postponing the details to Section \ref{sec:prelim}, we recall that it has been recently proved in \cite{dovetta} that in the critical regime a ground state exists if and only if the mass does not exceed a threshold value depending on the topology of the graph. More precisely, this threshold equals the value $\mu_{\R^+}$, the critical mass on the half-line, if $\G$ has at least a terminal edge (i.e., an edge ending into a vertex of degree 1, see Figure \ref{terminal}), whereas it equals $\mu_\rr$, the critical mass on the real line, if there is no edge of this kind (see for instance \cite{AST-cmp17} and Section \ref{sec:prelim} here for further details on $\mu_{\R^+}$ and $\mu_{\rr}$).

Turning to bound states, we point out that a detailed analysis of the stationary solutions and their stability properties has been so far carried out only  for the cubic NLS on {\em specific} graphs, such as the  dumbbell graph (see Figure \ref{bridge}, left)  in \cite{MP-amre16} (we remark that  the published paper contained an error that was later corrected on the arXiv version \cite{MP-amre16_c}). We  also note the work  \cite{goodman-rxv17} which analyzes  the NLS equation on  the dumbbell graph in relation with a discrete equation on the bowtie graph.

Our interest in this paper is on the contrary mainly devoted to a {\em generic} compact graph $\G$, with the intent of providing  information of variational nature on solutions of \eqref{stateq} that are constant on the whole graph. 

Up to multiplication by a phase, we may assume that these solutions are real and positive. So, letting $\ell:=|\G|$ be  the total length of the graph, the constant function $\cost$ given by

\begin{equation}
\label{cost}
\cost:=\sqrt{\frac{\mu}{\ell}}
\end{equation}
is always a solution of \eqref{EQ-def mass constraint}-\eqref{stateq} for a proper value of $\lambda$. We aim at understanding whether it is a ground state or not and if it is stable, both in the subcritical and in the critical regime.

We introduce the manifold
\[
\HmuG:=\{u\in H^1(\G)\,:\,\uLtwo^2=\mu\}
\]
of all functions in $H^1(\G)$ satisfying \eqref{EQ-def mass constraint}, and we note that $\cost\in\HmuG$.

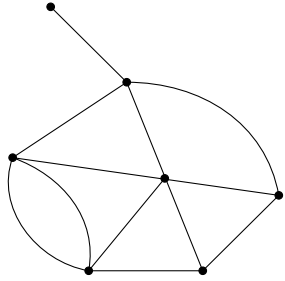
\begin{figure}[t]
	\centering
		\begin{tikzpicture}[xscale=0.5,yscale=0.5]
	
	\draw[fill] (-1,7) circle (0.1);
	\draw[fill] (-2,3) circle (0.1);
	\draw[fill] (0,0) circle (0.1);
	\draw[fill] (3,0) circle (0.1);
	\draw[fill] (5,2) circle (0.1);
	\draw[fill] (1,5) circle (0.1);
	\draw[fill] (2,2.45) circle (0.1);
	
	\draw (-1,7)--(1,5);
	\draw (1,5)--(-2,3);
	\draw (0,0) to [out=170,in=250] (-2,3);
	\draw (0,0) to [out=80,in=-20] (-2,3);
	\draw (0,0)--(3,0);
	\draw (3,0)--(1,5);
	\draw (3,0)--(5,2);
	\draw (5,2) to [out=100,in=0] (1,5);
	\draw (-2,3)--(5,2);
	\draw (0,0)--(2,2.45);
	
	\end{tikzpicture}
	\caption{a compact graph with a terminal edge}
	\label{terminal}
\end{figure}

Our first result provides an insight into the variational properties of $\kappa_\mu$, showing that, if the mass is small enough, it locally minimizes the energy in $\HmuG$ and it is stable, while this fails to be true when a threshold value of the mass is exceeded. 

\begin{thm}
\label{THM 1}
Let $p\in(2,6]$ and let $\G$ be a compact graph. Then there exists $\mu_1=\mu_1(\G,p)>0$ such that
\begin{itemize}
\item[(i)] if $\mu<\mu_1$, then $\cost$ is a local minimizer of $E$ in $\HmuG$ and it is orbitally stable;
\item[(ii)] if $\mu>\mu_1$, then $\cost$ is not a local minimizer of $E$ in $\HmuG$ and it is orbitally unstable.
\end{itemize}
\end{thm}

Pushing further the investigation, we show that $\cost$ is actually a ground state of $E$ provided that the mass is sufficiently small, possibly strictly smaller than the value $\mu_1(\G,p)$ identified in Theorem \ref{THM 1}.

\begin{thm}
\label{THM 2}
Let $p\in(2,6]$ and let $\G$ be a compact graph. Then there exists $\mu_2=\mu_2(\G,p)>0$ such that for every $\mu<\mu_2$, $\cost$ is a ground state of $E$ in $\HmuG$. 
\end{thm}

Let us highlight that both the previous theorems work  up to $p=6$, so that no difference between the subcritical and the critical regime seems to arise so far. However, in the subcritical regime, a natural notion of homothety deserves a few comments. Indeed, for $p\in(2,6)$, setting $\alpha=\frac{2}{6-p}$ and $\beta=\frac{p-2}{6-p}$, the normalized energy $\mu^{-2\beta-1}E(u)$ is invariant with respect to the transformations
\[
\G\mapsto t^{-\beta}\G,\qquad\qquad u(\cdot)\mapsto t^\alpha u(t^\beta\cdot),
\]
for every $u\in H^1(\G)$ and $t>0$. Note that such a scaling maps $\mu$ into $t\mu$ and any length in the graph, say $\ell$, into $t^{-\beta}\ell$. Therefore, it turns out that dealing with mass $\mu$ on $\G$ is equivalent to dealing with mass $t\mu$ on $t^{-\beta}\G$, as the product $\mu^\beta\ell$ is scale invariant. Hence, all the mass thresholds we identify in the subcritical regime should be properly rewritten as conditions on the term $\mu^\beta\ell$ (see remark \ref{REMARK-second eigenvalue and critical mass} below).

On the contrary, when $p=6$, this rigidity between the mass and the metric properties of the graph disappears. The critical setting requires a finer analysis of the actual value of the thresholds $\mu_1(\G,6)$ and $\mu_2(\G,6)$, due to the fact that the existence of the ground states is granted only for masses below either $\mu_{\R^+}$ or $\mu_{\rr}$, depending on $\G$. 

The following results establish lower bounds on $\mu_1(\G,6)$ based on  topological properties of $\G$. 

Recall that a graph is said to admit a \emph{cycle covering} if it can be covered by cycles (i.e. loops without edge repetitions, see \cite{AST-cmp17} and Figure \ref{cycles}).
\begin{figure}
	\centering
	\subfloat{\begin{tikzpicture}[xscale=0.4,yscale=0.4]
			
			\node at (2,0) {};
			\node at (4,0) {};
			
			\draw[fill] (2,0) circle (0.1);
			\draw[fill] (6,0) circle (0.1);
			\draw (0,0) circle (2);
			\draw (2,0)--(6,0);
			\draw (8,0) circle (2);
			
			\end{tikzpicture}} \qquad\qquad
	\subfloat{\begin{tikzpicture}[xscale=0.4,yscale=0.4]
			
			\draw[fill] (0,0) circle (0.1);
			\draw[fill] (1.5,2) circle (0.1);
			\draw[fill] (-2,1.5) circle (0.1);
			\draw[fill] (-0.5,3.5) circle (0.1);
			\draw[fill] (-0.25,1.75) circle (0.1);
			
			\draw (0,0)--(1.5,2);
			\draw (0,0)--(-2,1.5);
			\draw (0,0)--(-0.5,3.5);
			\draw (1.5,2)--(-2,1.5);
			\draw (1.5,2)--(-0.5,3.5);
			\draw (-2,1.5)--(-0.5,3.5);
			
			\draw[fill] (4.5,2) circle (0.1);
			\draw[fill] (6.5,4) circle (0.1);
			\draw[fill] (8.5,2) circle (0.1);
			\draw[fill] (6.5,0) circle (0.1);
			
			\draw (6.5,2) circle (2);
			
			\draw (4.5,2)--(1.5,2);
			\draw (4.5,2)--(6.5,4);
			\draw (4.5,2)--(6.5,0);
			
			\draw (8.5,2)-- (6.5,4);
			\draw (6.5,4)--(6.5,0);
			\draw (8.5,2)-- (6.5,0);
			
			\end{tikzpicture}}
	
	\caption{Examples of graphs with no terminal edge nor cycle covering}
	\label{bridge}
\end{figure}
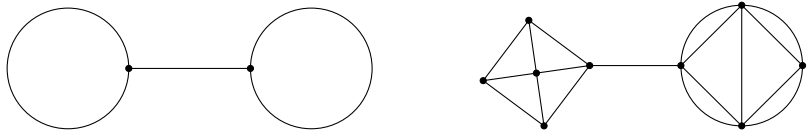

\begin{thm}
\label{THM 3}
Let $p=6$. Then, for every compact graph $\G$, $\mu_1(\G,6) \geq\frac{\pi}{2}$. Moreover, if $\G$ admits a cycle covering, then $\mu_1(\G,6)\geq\pi$.
\end{thm}

From a direct comparison between the above estimates and the values $\mu_{\R^+},\,\mu_{\rr}$ (reported  in  \eqref{massecrit}),  the next corollary can be easily proved.

\begin{cor}
\label{COR 1}
Let $p=6$, let $\G$ be a compact graph and write $\mu_1 = \mu_1(\G,6)$.
\begin{itemize}
\item[(i)] If $\G$ has a terminal edge, then for every $\mu\in(\mu_{\rr^+},\mu_1)$ there exists a solution of \eqref{EQ-def mass constraint}-\eqref{stateq} which is orbitally stable, i.e., $\cost$;
\item[(ii)] if $\G$ admits a cycle covering, then for every $\mu\in(\mu_\rr,\mu_1)$ there exists a solution of \eqref{EQ-def mass constraint}-\eqref{stateq} which is orbitally stable, i.e., $\cost$.
\end{itemize}
\end{cor}
Some comments are in order. First, Corollary \ref{COR 1} establishes the existence of stable solutions also in mass regimes where there is no ground state (the functional $E$ is unbounded from below on $H_\mu^1(\G)$ for $\mu>\mup$ or $\mur$), and therefore there is no obvious candidate for stability.

Secondly, combining Theorem \ref{THM 1} and  Corollary \ref{COR 1}, we obtain that 
when $\G$ has either a terminal edge or a cycle covering, the constant function $\cost$ is a local minimum of the energy in $\HmuG$ for all the masses at which ground states exist. This naturally raises the question whether $\cost$ is actually a ground state or not. In the second case this would mean that for masses
between $\mu_2(\G,6)$ and $\mup$ (or $\mur$, depending on the topology of $\G$) the ground state is {\em not} constant, though $\cost$ is always a local minimum of $E$. For this kind of graphs, up to now this is still an open problem (see Remark \ref{REM-constant ground states} below).
\medskip

For other classes of graphs, however, the preceding question can sometimes be answered, showing at the same time that, in addition to topology, metric properties of the graph can play a role too.
We therefore complete our analysis with the case of  graphs with no terminal edges nor cycle coverings. Under these assumptions, it is known that there always exists at least one {\em bridging edge}, whose removal disconnects the graph into two disjoint components. Moreover, as terminal edges are not allowed, both these connected components are different from a single vertex. 

Let  $\G_\ell$ denote a graph without terminal edges and cycle coverings, whose longest bridging edge has length $\ell$. We assume $\ell$ to be variable and
we describe the asymptotic  relation between $\ell$ and the threshold value $\mu_1(\G_\ell,6)$ of Theorem \ref{THM 1}.

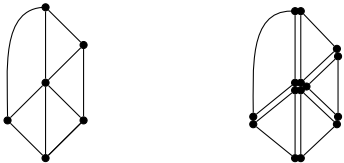
\begin{figure}[t]
	\centering
	\subfloat{
		\begin{tikzpicture}[xscale= 0.5,yscale=0.5]

		\node at (0,0) [nodo] (00) {};
		\node at (1,1) [nodo] (11) {};
		\node at (0,2) [nodo] (02) {};
		\node at (0,4) [nodo] (04) {};
		\node at (1,3) [nodo] (13) {};
		\node at (-1,1) [nodo] (-11) {};

		\draw[-] (00)--(11);
		\draw[-] (00)--(-11);
		\draw[-] (00)--(02);
		\draw[-] (-11)--(02);
		\draw[-] (02)--(11);
		\draw[-] (00)--(11);
		\draw[-] (04)--(13);
		\draw[-] (04)--(02);
		\draw[-] (13)--(11);
		\draw[-] (13)--(02);
		\draw (0,4) to [out=180, in=90] (-11);

		\end{tikzpicture}}\qquad\qquad\qquad
	\subfloat{\begin{tikzpicture}[xscale= 0.5,yscale=0.5]

		\node at (-1.2,1.2) [nodo] (d){};
		\node at (-1.2,1) [nodo] (f){};
		\node at (-.1,4) [nodo] (g){};
		\node at (-.1,2.1) [nodo] (h){};
		\node at (-.1,1.9) [nodo] (i){};
		\node at (-.1,.1) [nodo] (l){};
		\node at (.05,.1) [nodo] (n){};
		\node at (.05,1.9) [nodo] (o){};
		\node at (.2,2) [nodo] (p){};
		\node at (.05,2.1) [nodo] (q){};
		\node at (.05,4) [nodo] (r){};
		\node at (1,3) [nodo] (s){};
		\node at (1.03,2.8) [nodo] (t){};
		\node at (1.03,1.2) [nodo] (u){};
		\node at (1,1) [nodo] (v){};

		\draw (g) to [out=180, in=90] (d);

		\draw[-,] (d)--(h);
		\draw[-] (h)--(g);

		\draw[-] (f)--(l);
		\draw[-] (l)--(i);
		\draw[-] (i)--(f);
		
		\draw[-,] (r)--(s);
		\draw[-] (s)--(q);
		\draw[-] (q)--(r);
		
		\draw[-] (p)--(u);
		\draw[-] (u)--(t);
		\draw[-] (t)--(p);
		
		\draw[-] (v)--(o);
		\draw[-] (o)--(n);
		\draw[-] (n)--(v);

		\end{tikzpicture}}
	\caption{Example of a graph and a possible cycle covering of it}
	\label{cycles}
\end{figure}

\begin{thm}
\label{THM 4}
Let $p=6$. Then
\[
\lim_{\ell\to+\infty}\mu_1(\G_\ell,6)=\frac{\pi}{2}\qquad\text{and}\qquad\liminf_{\ell\to0}\mu_1(\G_\ell,6)\geq\pi.
\]
\end{thm}

To understand the meaning of Theorem \ref{THM 4}, recall from \cite{dovetta} that for graphs with no terminal edges nor cycle coverings, ground states exist if and only if the mass is smaller than or equal to $\mu_\rr$. Hence, we can exploit the relation between the metric of $\G_\ell$ and the value of $\mu_1(\G_\ell,6)$ unravelled in the previous theorem in two directions. 

On the one hand, considering graphs with a  long enough bridging edge, we prove that there exist a whole interval of masses for which ground states are nonconstant, thus answering in the affirmative the question raised after Corollary \ref{COR 1}. 
 
\begin{cor}
\label{COR 2}
Let $p=6$ and write $\mu_1 = \mu_1(\G_\ell,6)$. If $\ell$ is large enough, then there exists a whole interval of masses $\mu\in(\mu_1,\mu_\rr]$ such that the ground states of mass $\mu$ are not constant on $\G_\ell$.
\end{cor}

On the other hand, focusing on sufficiently small bridging edges, we recover an interval of masses beyond $\mu_\rr$ (the largest value of masses allowing for a ground state) where $\cost$ provides a stable solution of the NLS equation.

\begin{cor}
\label{COR 3}
Let $p=6$ and write $\mu_1 = \mu_1(\G_\ell,6)$. If $\ell$ is small enough, then, for every $\mu\in(\mu_\rr,\mu_1)$, there exists a solution of \eqref{EQ-def mass constraint}-\eqref{stateq}  which is orbitally stable, i.e., $\cost$.
\end{cor}

The paper is organized as follows. Section \ref{sec:prelim} states the framework we work within and collects some known results. In Section \ref{sec:local min} we begin the analysis of the variational properties of constant functions on graphs, proving Theorem \ref{THM 1}, whereas Section \ref{sec:small mass} provides the proof of Theorem \ref{THM 2}. Finally, in Section \ref{sec:critical regime} we deal in detail with the critical regime,  proving theorems \ref{THM 3}, \ref{THM 4} and their corollaries.

\section{Preliminaries}
\label{sec:prelim}

This section  overviews some preliminaries that serve as foundations for the analysis we develop in the forthcoming sections.

\subsection{Ground states on compact metric graphs}

We begin by recalling some known results concerning ground states of the NLS energy functional \eqref{EQ-def energy INTRO} with prescribed mass $\mu$ on general compact metric graphs.
The issue of the existence of these global minimizers has been dealt with in \cite{dovetta}.

We let
\[
\mathcal{E}_\G(\mu):=\inf_{u\in\HmuG}E(u)
\]
be the \textit{ground state energy level}, so that a ground state of $E$ of mass $\mu$ is a function $u \in H^1_\mu(\G)$ satisfying
\begin{equation}
\label{min}
E(u) = 	\mathcal{E}_\G(\mu)\,.
\end{equation}

On the one hand, in the subcritical case $p\in(2,6)$, ground states exist for every value of the mass $\mu$, regardless of the graph structure (\cite{dovetta}, Theorem 1.1).

On the other hand, in the critical case $p=6$  the topology of $\G$ becomes relevant. Indeed, Theorem 1.2 in \cite{dovetta} shows that  threshold phenomena occur, since for a critical value of the mass, $\mu_\G$, one has
\begin{itemize}
\item
$\mathcal{E}_\G(\mu)\begin{cases}
>-\infty & \text{if }\mu\leq\mu_\G\\
=-\infty & \text{if }\mu>\mu_\G,
\end{cases}
$
\item ground states with mass $\mu$ exist if and only if $\mu\leq\mu_\G$.
\end{itemize} 
Furthermore, recalling from  \cite{AST-cmp17}  that the critical masses of the half-line and of the line are explicitly given by 
\begin{equation}
\label{massecrit}
\mup = \frac{\pi \sqrt{3}}{4} \qquad\hbox{ and }\qquad \mur = \frac{\pi \sqrt{3}}{2},
\end{equation}
there results
\begin{itemize}
\item if $\G$ has at least one terminal edge (Figure \ref{terminal}), then $\displaystyle\mug=\mup$
\item if $\G$  has no terminal edge, then $\displaystyle\mug=\mur$.
\end{itemize}

Among the main tools that play a fundamental role in all the proofs of existence of ground states are
the Gagliardo-Nirenberg inequalities on compact graphs, that have been established in \cite{dovetta}. The first one reads
\begin{equation}
\label{EQ-GN subcritical}
\uLp^p\leq K_p \mu^{\frac{p+2}4} \|u\|_{H^1(\G)}^{\frac{p-2}{2}}
\end{equation}
and holds for every $u\in H_\mu^1(\G)$ and  every $p\geq 2$, with $K_p>0$ depending only on $\G$ and $p$.

When $p=6$, the preceding inequality is not powerful enough, and it has to be replaced by two modified versions,  derived in Proposition 4.1 and 4.2 of \cite{dovetta} (see also \cite{AST-cmp17}). If $\G$ is a graph with at least one terminal edge, then for every $\mu\leq\mu_{\rr^+}$ there exist $C>0$, depending only
on $\G$, such that for every $u\in\HmuG$,
\begin{equation}
\label{EQ-modified GN tip}
\uLsix^6\leq 3\Big(\frac{\mu-\theta}{\mu_{\rr^+}}\Big)^2\udot^2+C\theta^{\frac{1}{2}},
\end{equation}
with  $\theta\in[0,\mu]$ depending also on $u$.

If, on the contrary, $\G$ has no terminal edge, the previous inequality reads

\begin{equation}
\label{EQ-modified GN no tip}
\uLsix^6\leq 3\Big(\frac{\mu-\theta}{\mu_{\rr}}\Big)^2\udot^2+C\theta^{\frac{1}{2}}.
\end{equation}

\subsection{A bifurcation result}

We recall here a classical result from bifurcation theory that we are going to apply in Section \ref{sec:small mass}. It is taken from Chapter 5 in \cite{ambrosettiprodi}.

Let $X,Y$ be (real) Banach spaces and let $F\in C^2(\rr\times X,Y)$. Define, for $(\lambda, u) \in \rr\times X$,
\[
L(\lambda,u):= \frac{\partial F}{\partial u}(\lambda,u)
\]
and
\[
M(\lambda,u):= \frac{\partial^2 F}{\partial\lambda\partial u}(\lambda,u).
\]

The following theorem states a set of sufficient conditions on $F$ for the existence of a bifurcation point  and gives a local characterization of the set of zeros of $F$. 

\begin{thm}[\cite{ambrosettiprodi}, Theorem 4.1]
\label{bifurcation theorem}
Let $F\in C^2(\rr\times X,Y)$ satisfy 
\begin{equation*}
F(\lambda,0)=0\qquad\forall\lambda\in\rr\,.
\end{equation*}
Assume that there exist $\bar{\lambda}\in\rr$ and  $ v\in X\setminus\{0\}$ such that
\begin{itemize}
\item[(A1)]  $\quad\ker L(\bar{\lambda},0)=\{tv\,:\,t\in\rr\}$;
\item[(A2)] $\quad \rm{im }\, L(\bar{\lambda},0)$ is closed and has codimension 1;
\item[(A3)] $\quad M(\bar{\lambda},0)v\notin\rm{im}\, L(\bar{\lambda},0)$.
\end{itemize}
Then $(\bar{\lambda},0)$ is a bifurcation point for $F$. Precisely, there exists a neighborhood $\NN$ of $(\bar{\lambda},0)$ in $\rr\times X$ so that the set of nontrivial zeros of $F$ in $\NN$ is a unique $C^1$ curve through $(\bar{\lambda},0)$.
\end{thm}

Finally, we denote by  $\lambda_2(\G) $ the smallest positive eigenvalue of the Kirchhoff Laplacian on $\G$ (i.e. $-(\,\cdot\,)''$ on $\G$, coupled with Kirchhoff conditions) namely
\begin{equation}
\label{la2}
\lambda_2(\G) = \inf_{\varphi \in H^1(\G)\setminus\{0\}\atop \int_\G \varphi \dx =0} \frac{ \int_\G |\varphi' |^2\dx}{\int_\G |\varphi |^2\dx}.
\end{equation}
For future reference we recall from Theorem 1 in \cite{friedlander} (see also \cite{nicaise} and \cite{KN14}), that  for every $\G$ such that $|\G|=\ell$,
\begin{equation}
\label{estla2}
\lambda_2(\G)\geq\frac{\pi^2}{\ell^2}.
\end{equation}

\subsection{Well-posedness}

Before discussing  the orbital stability we  address the problem of  the well-posedness of equation \eqref{nlse} for initial data $u_0\in H^1(\G)$. As for the usual NLS equation on the real line we need to distinguish the subcritical and critical regime. In the subcritical regime global well-posedness holds true for every initial datum in $H^1(\G)$. In the critical regime, instead, it  holds   for initial data with mass $\mu$ smaller than $\mu_{\mathbb R^+}$ ($\mu_{\mathbb R}$) if $\G$ has at least one terminal edge (has no terminal edge). 

Both in the subcritical and critical case the strategy to prove global well-posedness is the usual one: prove local existence of solutions in $H^1(\mathcal G)$ for $t$ in an interval $(0,T^*)$; show that the conservation laws 
\[
\|u(t)\|_2 = \|u_0\|_2  \qquad \textrm{(mass) }
\]
and 
\[
E(u(t)) = E(u_0) \qquad\qquad\qquad\textrm{(energy)}
\]
are satisfied by the solution $u(t)$ and for $t\in(0,T^*)$; use energy and mass  conservation to show that the $H^1$-norm of $u(t)$ must stay bounded for all times, implying that $T^*=+\infty$.

For the proof of the local existence of solutions   and of the conservation laws we refer to \cite{CFN-non17}, as the argument used there for  non-compact graphs adapts without changes  to compact graphs. 

We focus our attention on the third  step, proving a uniform bound for $\|u(t)\|_{H^1}$. 

We first consider  the critical case $p=6$ for a graph having at least one terminal edge. Assume that  the initial datum has mass $\|u_0\|_2^2 = \mu$.  By the conservation of the energy and by the modified Gagliardo-Nirenberg inequality \eqref{EQ-modified GN tip}  we infer, for all $t \in (0,T^*)$,
\[
\begin{aligned}
E(u_0)  = E(u(t))  = & \frac{1}{2}\|u'(t)\|_2^2-\frac{1}{6}\|u(t)\|_6^6 \\ 
\geq &  \frac{1}{2}\|u'(t)\|_2^2 -  \frac12 \Big(\frac{\mu-\theta}{\mu_{\rr^+}}\Big)^2\|u'(t)\|_2^2 - C \theta^{\frac{1}{2}} \\ 
\geq & \frac{1}{2} \left( 1 -  \frac{\mu^2}{\mu_{\rr^+}^2}\right)\|u'(t)\|_2^2 - C \mu^{\frac{1}{2}}.
\end{aligned}
\]
Hence, if $\mu<\mu_{\rr^+}$, one has $\|u'(t)\|_2 \leq C$ and, by mass conservation,  $\|u(t)\|_{H^1}\leq C$. Global well-posedness then follows by a standard  Gronwall-type argument, see \cite{CFN-non17} for more details. 

If $\G$ has no terminal edge, using the same argument together with inequality \eqref{EQ-modified GN no tip} one obtains global well-posedness for $\mu<\mu_{\rr}$. 

In the subcritical case, $p\in (2,6)$,  by using the standard Gagliardo-Nirenberg inequality \eqref{EQ-GN subcritical} and mass and energy conservation, it follows that  
\[
E(u_0)  = E(u(t))  \geq  \frac{1}{2} \|u'(t)\|_2^2 - \frac{K_p}{p}  \mu^{\frac{p}4+\frac12} \|u'(t)\|_2^{\frac{p}2-1}\,,
\]
which implies that $\|u'(t)\|_2\leq C$ for all $\mu$, allowing one to conclude as above.

We remark that, as far as orbital stability is concerned, only local existence is needed, see Assumption 1 in \cite{GSS}. The discussion on the global well-posedness  has to be  understood as an additional result which may have an interest on its own. 

\subsection{Orbital stability}
\label{subsec: stability}

We note that  for every $\lambda>0$ the function $\widetilde\kappa = \widetilde \kappa(\lambda)=\lambda^{\frac{1}{p-2}}$ is a solution of the stationary equation \eqref{stateq}. For $\lambda = (\mu/\ell)^{p/2-1}$ the function  $\widetilde \kappa(\lambda)$ coincides with the constant solution $\cost$ of mass $\mu$ given in   \eqref{cost}. 

To prove orbital stability of  the stationary solution $\kappa_\mu e^{i\lambda t}$ we follow the general approach in \cite{GSS}. Define the functional $Q(u) := \frac12\|u\|_2^2$, and consider the operator  
\[H_\lambda  := E''(\widetilde \kappa) + \lambda Q''(\widetilde \kappa). \]
It turns out that $H_\lambda$, to be understood as an operator in $L^2(\G)\oplus L^2(\G)$, is given by 
\[
H_\lambda = \begin{pmatrix} L_1 & 0 \\ 0 & L_0  \end{pmatrix}
\]
where $L_0$ and $L_1$ are the self-adjoint operators 
\[
L_0 u = - u'' \qquad u \in  H_K^2(\G) 
\]
\[
L_1 u = - u'' - (p-2)\lambda u \qquad u \in  H_K^2(\G) .
\]
To study the orbital stability of $\widetilde \kappa$ one is interested in the spectral properties of $H_\lambda$. We note that $L_0$ has zero as a simple eigenvalue and the remaining part of the spectrum is positive.  On the other hand $L_1$ has at least one negative eigenvalue, $- (p-2)\lambda$, the corresponding eigenfunction being  the constant function. 

Assume now that  $\widetilde \kappa$ is a local minimizer of the energy functional constrained to $H^1_\mu(\G)$ for some value of $\mu$; then it is also a local minimizer for  $E(u) + \lambda Q(u)$ under the same mass constraint. Hence, by \cite[Prop. B1]{FGJS-cmp04},  $H_\lambda$ has at most one negative eigenvalue.  By the  discussion above we conclude that  $H_\lambda$ has exactly one negative eigenvalue.  Hence, by Theorem 3 in \cite{GSS} a local minimizer of the energy functional constrained to $H^1_\mu(\G)$  is orbitally stable. 

We will use this property in Section \ref{sec:local min} to prove that for $\mu$ smaller than the threshold value $\mu_1$,  $\cost$  is orbitally stable.

\section{Local minimality and stability: proof of Theorem \ref{THM 1}}

\label{sec:local min}

This section is devoted to the proof of Theorem \ref{THM 1} which will be carried out by evaluating the sign of the second derivative of the energy at $\cost$.

Therefore we begin by computing the second derivative of $E$ at $u = \cost$, considered as a functional $\overline E$ on the manifold $H^1_\mu(\G)$, namely the $L^2$ sphere of radius $\sqrt\mu$ in $H^1(\G)$. Note that the tangent space to $H^1_\mu(\G)$ at $\cost$ is 
\[
T_{\cost} H^1_\mu(\G) = \{\varphi \in H^1(\G) \; :\; \int_\G{\rm Re} (\varphi) \dx = 0 \}.
\]

\begin{proposition}
\label{seconder}
For every $p\in (2,6]$ there results 
\begin{equation}
\label{secder}
\overline E''(\cost)\varphi^2 =  \int_\G|\varphi'|^2 \dx - (p-2)\cost^{p-2} \int_\G \big| {\rm Re}(\varphi)\big|^2 \dx,
\end{equation}
for every $\varphi \in T_{\cost} H^1_\mu(\G)$.
\end{proposition}

\begin{proof} We take $\varphi \in T_{\cost} H^1_\mu(\G)$ and we define a smooth curve $\gamma : (-1,1) \to H^1_\mu(\G)$ as
\[
\gamma(t) = \frac{\sqrt{\mu}}{\|\cost+t\varphi\|_2}(\cost + t\varphi).
\]
Note that $\gamma (0) = \cost$.

To simplify notation, let $g(t):=\sqrt{\mu}/\|\cost+t\varphi\|_{2}$, so that $\gamma(t) = g(t)(\cost + t\varphi)$. With elementary computations keeping into account that $ {\rm Re}(\varphi)$ has zero mean value,
\begin{align*}
g(0) & =1\\
\dot g(0) & =-\frac{1}{\mu}\int_\G {\rm Re }(\cost\varphi)\,dx=0\\
\ddot g(0) & =\frac{3}{\mu^2}\Big(\int_\G {\rm Re }(\cost\varphi)\,dx\Big)^2-\frac{1}{\mu}\int_\G|\varphi |^2\,dx = -\frac{1}{\mu}\int_\G|\varphi |^2\,dx.
\end{align*}
Therefore
\begin{align*}
\label{gamma}
\dot \gamma(0) & = \dot g(0) \cost + g(0) \varphi =  \varphi \\
\ddot \gamma(0) & = \ddot g(0) \cost + 2\dot g(0) \varphi = -\left(\frac{1}{\mu} \int_\G|\varphi|^2 \dx\right) \cost.
\end{align*}
Now
\[
\overline E''(\cost)\varphi^2 = \frac{d^2}{dt^2} E(\gamma(t))\Big|_{t=0} =  E''(\cost)\dot\gamma(0)^2 + E'(\cost) \ddot \gamma(0).	
\]
Since for every $v\in H^1(\G)$,
\[
E'(\cost)v = -\cost^{p-1}\int_\G   {\rm Re}(v) \dx 
\] 
and 
\[
E''(\cost)v^2 = \int_\G |v'|^2\dx  - (p-1)\cost^{p-2} \int_\G \big|{\rm Re}(v)\big|^2\dx   - \cost^{p-2} \int_\G \big|{\rm Im}(v)\big|^2\dx,
\]
by virtue of the expressions of $\dot\gamma(0)$ and $\ddot\gamma(0)$ we conclude that
\begin{align*}
\overline E''(\cost)\varphi^2 & = \int_\G |\varphi' |^2\dx - (p-1)\cost^{p-2}\int_\G \big| {\rm Re}(\varphi)\big|^2\dx- \cost^{p-2}\int_\G \big| {\rm Im}(\varphi)\big|^2\dx + \cost^p \frac\ell\mu\int_\G |\varphi|^2 \dx \nonumber \\ 
& =  \int_\G |\varphi' |^2\dx - (p-2)\cost^{p-2}\int_\G \big| {\rm Re}(\varphi)\big|^2\dx,
\end{align*}
having also used $\cost = \sqrt{\mu/\ell}$. This concludes the proof. 
\end{proof}

\noindent{\em Proof of Theorem \ref{THM 1}}.
Recalling from \eqref{la2} the  definition of $\lambda_2(\G)$ and using it in \eqref{secder}  we obtain
\[
\begin{aligned}
\overline{E}''(\cost)\varphi^2\geq & \lambda_2(\G)\int_\G|{\rm Re}(\varphi)|^2\,dx-(p-2)\cost^{p-2}\int_\G|{\rm Re}(\varphi)|^2\,dx + \int_\G |{\rm Im}(\varphi') |^2\dx \\ 
= & \Big(\lambda_2(\G)-(p-2)\cost^{p-2}\Big)\int_\G|{\rm Re}(\varphi)|^2\,dx +  \int_\G |{\rm Im}(\varphi') |^2\dx
\end{aligned}
\] 
for every $\varphi \in T_{\cost} H^1_\mu(\G)$.
Then, if  
\[
\cost^{p-2} < \frac{\lambda_2(\G)}{p-2},
\]
i.e, if
\begin{equation}
\label{mu2}
\mu < \ell\left(  \frac{\lambda_2(\G)}{p-2}\right)^{\frac2{p-2}} =: \mu_1,
\end{equation}
we see that $\overline E''(\cost)\varphi^2 >0$ for every $\varphi \in T_{\cost} H^1_\mu(\G)$, unless $\varphi = i c$ for some constant $c$; in the latter case one has $\overline E''(\cost)\varphi^2 = 0$.  However, by the mass constraint, every variation of $\cost$ which keeps the function constant must reduce to a  phase multiplication which does not affect the energy, and therefore $\cost$ is a local minimum for $E$ on $H^1_\mu(\G)$.

On the other hand, let $\varphi_2$ be an eigenfunction corresponding to $\lambda_2(\G)$, so that
\[
\int_\G |\varphi_2' |^2\dx = \lambda_2(\G) \int_\G |\varphi_2 |^2\dx.
\]
Assume, without loss of generality,  that $\varphi_2$ is real valued. Testing $\overline E''(\cost)$ with $\varphi_2$ we obtain
\begin{equation*}
\overline E''(\cost)\varphi_2^2 =  \Big(\lambda_2(\G) -(p-2) \cost^{p-2}\Big )\int_\G |\varphi_2 |^2\dx
\end{equation*}
which shows, via \eqref{mu2}, that $\overline{E}''(\cost)\varphi_2^2<0$ whenever $\mu>\mu_1$. In this case, $\cost$ is no longer a local minimum of the energy in $\HmuG$.

The stability properties of $\cost$  follow then directly from the previous part of the proof and the discussion in Section \ref{subsec: stability}.
\hspace{\stretch{1}} $\Box$ 

\begin{rem}
\label{REM-mu1 in THM1}
The previous argument does not allow us to conclude at the threshold value $\mu=\mu_1$. Indeed in this case $\overline{E}''(\cost)$ is only positive semidefinite  on $T_{\cost}\HmuG$, which is not enough to decide whether $\cost$ locally minimizes the energy or not.
\end{rem}

\section{Global minimality: proof of Theorem \ref{THM 2}}
\label{sec:small mass}

The aim of this section is to deal with ground states, proving that, both in the subcritical and in the critical case, the constant function $\cost$ is a global minimizer of the energy in $H_\mu^1(\G)$, provided the mass $\mu$  is sufficiently small. To establish this result, as stated in Theorem \ref{THM 2}, we apply Theorem \ref{bifurcation theorem}. Thus we define the spaces of real-valued functions
\[
X = \{ u \in H^2_K(\G) \; : \;  {\rm Im } (u) = 0\} ,\qquad Y = \{ u \in L^2(\G) \; : \;  {\rm Im } (u) = 0\}
\]
and a map $F:\rr\times X\to Y$ as

\begin{equation}
\label{EQ- def bifurcation F}
F(\lambda,u):=u''+|u|^{p-2}u-\lambda u.
\end{equation}
Note that 
\[
F(\lambda,0)=0
\]
for every $\lambda\in\rr$; these are the trivial zeros of $F$. We begin with the following preliminary result.

\begin{proposition}
\label{LEMMA- bifurcation}
There exists a neighborhood $\NN$ of $(0,0)$ in $\rr\times X$ such that the set of nontrivial   zeros of $F$ in $\NN$ is a unique $C^1$ curve through $(0,0)$. Moreover,  these zeros are of the form $(\lambda,\pm\lambda^{\frac1{p-2}})$, for sufficiently small $\lambda>0$. 
\end{proposition}

\begin{proof} We prove the statement checking that $F$ satisfies the hypotheses of Theorem \ref{bifurcation theorem} at
$(0,0)$. Clearly $F$ is $C^2$ and $F(0,0) = 0$. Moreover, differentiating \eqref{EQ- def bifurcation F} with respect to $u$, we have
\begin{equation}
\label{EQ-def F' u}
L(\lambda,u)v:=  \frac{\partial F}{\partial u}(\lambda,u)v=v''+(p-1)|u|^{p-2}v-\lambda v \qquad\forall v\in X
\end{equation}
and, evaluating at $(0,0)$,
\[
L(0,0)v=v''.
\]
Therefore $v\in \ker L(0,0)$ if and only if $v'' =0$ in $\G$. Now multiplying by $  v$ and integrating on $\G$ yields, by self-adjointness, 
\[
0  = \int_\G   v'' v \dx =   - \int_\G |v'|^2\dx.
\]
This, and the continuity of $v$, show that
\begin{equation*}
\ker L(0,0)=\{v\in X\,:\, v\text{ is constant on }\G\}
\end{equation*}
and assumption $(A1)$ is fulfilled.
	
We now prove that
\begin{equation*}
{\rm im}\,L(0,0)=\Big\{f\in Y\,:\,\int_{\G}f\,dx=0\Big\},
\end{equation*}
which is closed and has codimension 1. 
	
Indeed, if $L(0,0)v=f$, for some $f\in Y$ and some $v\in X$, i.e. if $v'' = f$, then integrating on $\G$  gives
\begin{equation*}
0=\int_{\G}v''\,dx=\int_{\G}f\,dx.
\end{equation*}
On the other hand, for every $f\in Y$ such that $\int_{\G}f\,dx=0$, it is easily seen that the minimization problem
\begin{equation*}
\inf_{v\in H^1(\G)}\int_\G\frac{1}{2}|v'|^2+f v\,dx
\end{equation*}
(here $H^1(\G)$ is the space of real-valued $H^1$ functions) always admits  a solution, granting the existence of a function $v\in X$ satisfying $v''=f$. Hence, $(A2)$ is verified too.
	
We finally check that $(A3)$ holds. Differentiating once more \eqref{EQ-def F' u} with respect to $\lambda$ and evaluating at $(0,0)$, we see that
\[
M(0,0)v:= \frac{\partial^2 F}{\partial \lambda\partial u}(0,0)v = - v,\qquad v\in X,
\]
so that, if $v\in \ker L(0,0)$ and $v\not\equiv 0$, then $M(0,0)v\notin\text{im } L(0,0)$, proving that $(A3)$ holds.
	
Therefore, Theorem \ref{bifurcation theorem} applies to $F$ at $(0,0)$, stating that there exists a neighborhood $\NN$ of $(0,0)$ in $\rr\times X$ such that the nontrivial zeros of $F$ that belong to $\NN$ describe a unique $C^1$ curve through $(0,0)$. Since the points
\[
(\lambda,\pm \lambda^{\frac{1}{p-2}})
\]
are zeros of $F$ for every $\lambda>0$, and they enter $\NN$ as $\lambda\to0$, we conclude.
\end{proof}

We can now provide the proof of Theorem \ref{THM 2}.
\medskip

\noindent\textit{Proof of Theorem \ref{THM 2}.} We split the proof into two parts, depending on the value of the exponent of the nonlinearity.
\medskip

\noindent\textit{Part (i): the subcritical case $p\in(2,6)$.} Let $\{u_n\}_{n\in\mathbb{N}}\subset H^1(\G)$ be a sequence of solutions to the minimum problem  \eqref{min} such that

\[
\mu_n:=\|u_n\|_2^2\to 0\qquad\text{as }n\to+\infty\,.
\]
Up to multiplication by a phase, we may assume that each $u_n$ is real (and positive).

Of course, each $u_n$ satisfies 
\begin{equation}
\label{eqn}
u_n''+|u_n|^{p-2}u_n -\lambda_n u_n=0\qquad\hbox{on } \G
\end{equation}
for some Lagrange multiplier $\lambda_n$.
We prove that, for $n$  large enough, the points $(\lambda_n,u_n)$ necessarily enter the neighborhood $\NN$ of Proposition \ref{LEMMA- bifurcation}.

Since for every $n\in\mathbb{N}$,
\begin{equation}
\label{EQ-minimum energy negative}
E(u_n)=\mathcal{E}_\G(\mu_n) \le E(\kappa_{\mu_n})<0,
\end{equation}
using \eqref{EQ-GN subcritical} leads to
\begin{align}
\label{EQ- bound with GN subcritical}
\|u_n'\|_2^2 & <\frac{2}{p}\|u_n\|_p^p\leq C\mu_n^{\frac{p+2}4} \left(\mu_n+\|u_n'\|_2^2\right)^{\frac{p-2}4} \nonumber \\
& \le C\mu_n^{\frac{p}2} + C \mu_n^{\frac{p+2}4}\|u_n'\|_2^{\frac{p-2}{2}},
\end{align}
for some $C > 0$ depending only on $\G$ and $p$. As for $p\in(2,6)$ we have $\frac{p-2}{2} <2$, and $\mu_n \to 0$, we readily see that $\|u_n'\|_2\to 0$ as $n\to \infty$. Therefore $u_n\to 0$ in $H^1(\G)$ and in $L^\infty(\G)$. 

Next, multiplying \eqref{eqn} by $u_n$ and integrating over $\G$, we see, using \eqref{EQ- bound with GN subcritical}, that
\begin{align}
\label{est}
|\lambda_n | & \le \frac{\int_\G|u_n'|^2\,dx+\int_\G|u_n|^p\,dx}{\mu_n}<\frac{(2/p+1)\int_\G|u_n|^p\,dx}{\int_\G|u_n|^2\,dx} \nonumber \\
& \le   C \frac{\mu_n^{\frac{p}2} +  \mu_n^{\frac{p+2}4}\|u_n'\|_2^{\frac{p-2}{2}}}{\mu_n} = C\Big(\mu_n^{\frac{p-2}2} + \mu_n^{\frac{p-2}4}\|u_n'\|_2^{\frac{p-2}{2}}\Big) = o(1)
\end{align}
as $n\to \infty$.

Finally, from \eqref{eqn}, \eqref{est} and $u_n\to 0$ in $L^\infty$, we obtain that $u_n\to 0$ in $H_K^2(\G)$.
We have thus proved that 
\[
(\lambda_n, u_n)\to(0,0)\quad\text{ in } \rr \times H_K^2(\G) \qquad\text{ as }n\to+\infty.
\]
This means (keeping into account that each $u_n$ is real-valued) that  when $n$ is large enough, namely when $\mu$ is less than a threshold value $\mu_2$, the point $(\lambda_n,u_n)$ is in the neighborhood $\NN$ of Proposition \ref{LEMMA- bifurcation}. Hence $u_n$ is constant on $\G$, i.e. $u_n = {\cost}_n$, and we conclude.
\medskip

\noindent\textit{Part (ii): the critical case $p=6$.} The argument is the same as the one of part \textit{(i)}, the only difference being that we can no longer make use of inequality \eqref{EQ-GN subcritical} to guarantee the boundedness in $H^1(\G)$ of a sequence of ground states.

In order to recover this step, let us first assume that $\G$ has at least one terminal edge. Let
again $\{u_n\}_{n\in \nn} \subset H^1(\G)$ be a sequence of solutions of the minimum problem  \eqref{min} with $\|u_n\|_2^2:=\mu_n\to0$ as $n\to+\infty$. Combining \eqref{EQ-minimum energy negative} with the modified Gagliardo-Nirenberg inequality \eqref{EQ-modified GN tip},  leads to
\[
\|u_n'\|_2^2<\frac{1}{3}\|u_n\|_6^6\leq\Big(\frac{\mu_n-\theta_n}{\mu_{\rr^+}}\Big)^2\|u_n'\|_2^2+C\theta_n^{\frac{1}{2}}
\]
and, since $\mu_n\to0$ and $\theta_n\in[0,\mu_n]$, this implies $\|u_n'\|_2\to0$ as $n\to+\infty$, allowing the argument developed in the subcritical case to apply again.

When $\G$ has no terminal edge one can repeat the previous passages in the same way, simply replacing
inequality \eqref{EQ-modified GN tip} with \eqref{EQ-modified GN no tip} when needed.
\hspace{\stretch{1}} $\Box$

\section{The critical regime: proof of Theorems \ref{THM 3} and \ref{THM 4}}

\label{sec:critical regime}

In the critical case $p=6$ the situation is more involved than in the subcritical setting. Indeed, in view of Section \ref{sec:prelim}, ground states  exist if and only if the mass does not exceed a critical value $\mu_\G$ that takes either the value $\mur = \pi \sqrt3/2$ or the value $\mup = \pi\sqrt3/4$, depending on the topology of $\G$. Hence, a natural question is whether it is possible to locate  the threshold $\mu_1$ introduced in Theorem \ref{THM 1} with respect to $\mu_{\rr^+}$ and $\mu_\rr$. A first answer to this problem is provided by Theorem \ref{THM 3}, that we prove now.
\medskip

\noindent\textit{Proof of Theorem \ref{THM 3}.} The first part of the result is straightforward. Indeed,  by \eqref{mu2} and estimate \eqref{estla2} we immediately obtain
\[
\mu_1(\G,6)\geq\frac{\pi}{2}\,.
\]
for every $\G$.

Assume now that  $\G$ admits a cycle covering (Figure \ref{cycles}).
Pick any $\phi\in H^1(\G)$ with zero mean value and let $\phi_1,\phi_2$ be its real and imaginary part, respectively, so that $\phi=\phi_1+i\phi_2$. Of course, both $\phi_1$ and $\phi_2$ have zero mean value. Moreover, for $i=1,2$, denote by $\phi_i^+$ and $\phi_i^-$ the positive and negative parts of $\phi_i$, so that $\phi_i = \phi_i^+ - \phi_i^-$.

Since $\G$ can be covered by cycles, it follows that almost every value $t$ in the range of $\phi_i$, $i=1,2$, has at least two preimages on $\G$, namely
\[
\symbol{35}\{x\in\G\,:\,\phi_i(x)=t\}\geq2\qquad\text{for a.e. }t\in\phi_i(\G),
\]
and the same holds for $\phi_i^+$ and $\phi_i^-$.

For $i=1,2$, let $\ell_i^\pm = |\{x \in \G\; : \; \phi_i^\pm >0\}|$, and note that $\ell_i^+ + \ell_i^- \le \ell = : |\G|$.
Now take the  symmetric rearrangements $\widehat{\phi_i^\pm }\in H^1(-\ell_i^\pm/2,\ell_i^\pm/2)$ of $\phi_i^\pm$, defined as in \cite{AST2015}. Clearly $\widehat{\phi_i^+}(-\ell_i^+/2) = \widehat{\phi_i^+}(\ell_i^+/2) = 0$, and likewise for $\widehat{\phi_i^-}$. Finally, define two functions $\psi_i: [0,\ell] \to \rr$ as
\[
\psi_i(x) = \begin{cases} \widehat{\phi_i^+}(x-\ell_i^+/2) & \text{ if } x \in [0,\ell_i^+] \\ 
-\widehat{\phi_i^-}(x-\ell_i^+ -\ell_i^-/2) & \text{ if } x \in [\ell_i^+, \ell_i^+ + \ell_i^-] \\
0 & \text{ if } x \in [\ell_i^+ + \ell_i^-,\ell]. \end{cases}
\]
By the standard properties of rearrangements (see Proposition 3.1 in \cite{AST2015}), $\psi_1,\psi_2 \in H^1_0(0,\ell)$, have zero mean value and satisfy the usual relations
\[
\int_0^\ell |\psi_i|^2\dx = \int_\G|\phi_i|^2\dx \qquad\text{ and }\qquad 
\int_0^\ell |\psi_i'|^2\dx \le \int_\G |\phi_i'|^2\dx.
\]
Therefore, if we consider $\psi:=\psi_1+i\psi_2$, then $\psi\in H_0^1(0,\ell)$, has zero mean value and
\[
\frac{\int_\G|\phi'|^2\,dx}{\int_\G|\phi|^2\,dx}\geq\frac{\int_0^\ell |\psi'|^2\,dx}{\int_0^\ell |\psi|^2\,dx}
\]
so that, passing to the infimum over all $\phi\in H^1(\G)$ with zero mean value, we obtain

\begin{equation}
\label{EQ-eigen geq circle}
\lambda_2(\G) = \inf_{\substack{\phi \in H^1(\G)\setminus\{0\}\\ \int_\G \phi\,dx=0}}\frac{\int_\G| \phi '|^2\,dx}{\int_\G |\phi|^2\,dx} \geq\inf_{\substack{\psi \in H^1_0(0,\ell)\setminus\{0\}\\ \int_0^\ell \psi\,dx=0}}\frac{\int_0^\ell| \psi '|^2\,dx}{\int_0^\ell |\psi|^2\,dx}=\frac{4\pi^2}{\ell^2}
\end{equation}
(see also \cite{bandlevy} and \cite{berkolaiko-mugnolo} for further references on this inequality and several other of similar fashion).

Inserting this in \eqref{mu2} we conclude that
\[
\mu_1(\G,6)\geq\pi.
\]
\hspace{\stretch{1}} $\Box$

\begin{rem}
\label{REMARK-second eigenvalue and critical mass}
Estimates for $\mu_1$ similar to the ones in Theorem \ref{THM 3} can be derived in the same way also in the subcritical regime. Indeed, combining \eqref{estla2} with \eqref{mu2}, we obtain, for every compact graph $\G$ of length $\ell$ and every $p\in(2,6]$,
\[
\mu_1(\G,p)\geq\ell^{\frac{p-6}{p-2}}\Big(\frac{\pi^2}{p-2}\Big)^{\frac{2}{p-2}}\,.
\]
Note that, when $p\in(2,6)$, rearranging terms, condition \eqref{mu2} may  be rewritten as
\[
\mu^{\frac{p-2}{6-p}}\ell < \Big(\frac{\pi^2}{p-2}\Big)^{\frac{2}{6-p}},
\]
which is consistent with what we anticipated in the Introduction, since the term $\mu^{\frac{p-2}{6-p}}\ell$ is scale invariant.
	
Analogously, if $\G$ has a cycle covering, then by \eqref{EQ-eigen geq circle} it follows
\[
\mu^{\frac{p-2}{6-p}}\ell  > \Big(\frac{4\pi^2}{p-2}\Big)^{\frac{2}{p-2}}\,.
\]
	
\end{rem}
\medskip

\noindent\textit{Proof of Corollary \ref{COR 1}.} It is straightforward. By Theorem \ref{THM 1} the function $\cost$ is orbitally stable for every $\mu$ in the interval $(0,\mu_1)$, $\mu_1 = \mu_1(\G,6)$.

If $\G$ has a terminal edge,  by Theorem \ref{THM 3},
\[
\mu_1 \ge \frac\pi{2} > \frac{\pi\sqrt 3}4= \mup,
\]
while if $\G$ has a cycle  covering, again by Theorem \ref{THM 3},
\[
\mu_1 \ge \pi> \frac{\pi\sqrt 3}2= \mur.
\]

\hspace{\stretch{1}} $\Box$ 
\medskip

We now turn our attention to Theorem \ref{THM 4} and Corollaries \ref{COR 2}-\ref{COR 3}. From now on, we consider compact graphs with no terminal edge nor cycle covering (Fig. \ref{bridge}). Graphs like these always have at least one bridging edge, whose removal disconnects the graph into two disjoint connected components, each of them different from a single vertex.
For the sake of simplicity, in what follows we assume that there is exactly one of such edges in the graph. However, this does not lead to any loss of generality, since all the arguments we develop below can be easily adapted to the case of multiple bridges, considering the longest one among them.

\begin{figure}
\begin{center}
\begin{tikzpicture}[xscale= 0.7,yscale=0.7]
\node at (1,0) [nodo] (01) {};
\node at (2,0) [nodo] (02) {};
\node at (7,0) [nodo] (03) {};
\node at (0,1) [nodo] (11) {};
\node at (1.5,1) [nodo] (12) {};
\node at (3,1) [nodo] (13) {};
\node at (6,1) [nodo] (14) {};
\node at (7,1) [nodo] (15) {};
\node at (8,1) [nodo] (16) {};
\node at (1,2) [nodo] (21) {};
\node at (2,2) [nodo] (22) {};
\node at (7,2) [nodo] (23) {};

\node at (4.5,1.4) (ell) {$\scriptstyle\ell$};
\node at (4.5,-.3) (ell) {$\G_\ell$};
\node at (-.2,1.7) (ell) {$\scriptstyle\Gamma_1$};
\node at (8.3,1.7) (ell) {$\scriptstyle\Gamma_2$};

\node at (11.9,1.7) (ell) {$\scriptstyle\Gamma_1$};
\node at (17.1,1.7) (ell) {$\scriptstyle\Gamma_2$};
\node at (15,-.3) (ell) {$\G_0$};

\draw[-] (11)--(01);
\draw[-] (01)--(02);
\draw[-] (02)--(13);
\draw[-] (13)--(22);
\draw[-] (22)--(21);
\draw[-] (21)--(11);
\draw[-] (11)--(12);

\draw[-] (02)--(21);
\draw[-] (01)--(22);
\draw[-] (03)--(23);

\draw[-] (14)--(03);
\draw[-] (03)--(16);
\draw[-] (16)--(23);
\draw[-] (23)--(14);
\draw[-] (15)--(16);

\draw[-] (13)--(14);

\node at (13,0) [nodo] (001) {};
\node at (14,0) [nodo] (002) {};
\node at (16,0) [nodo] (003) {};
\node at (12,1) [nodo] (011) {};
\node at (13.5,1) [nodo] (012) {};
\node at (15,1) [nodo] (013) {};
\node at (15,1) [nodo] (014) {};
\node at (16,1) [nodo] (015) {};
\node at (17,1) [nodo] (016) {};
\node at (13,2) [nodo] (021) {};
\node at (14,2) [nodo] (022) {};
\node at (16,2) [nodo] (023) {};

\draw[-] (011)--(001);
\draw[-] (001)--(002);
\draw[-] (002)--(013);
\draw[-] (013)--(022);
\draw[-] (022)--(021);
\draw[-] (021)--(011);
\draw[-] (011)--(012);

\draw[-] (002)--(021);
\draw[-] (001)--(022);
\draw[-] (003)--(023);

\draw[-] (014)--(003);
\draw[-] (003)--(016);
\draw[-] (016)--(023);
\draw[-] (023)--(014);
\draw[-] (015)--(016);

\draw[-] (013)--(014);

\end{tikzpicture}
\caption{\footnotesize{The graphs $\G_\ell$ and their ``limit'' $\G_0$}}
\label{gl}
\end{center}

\end{figure}
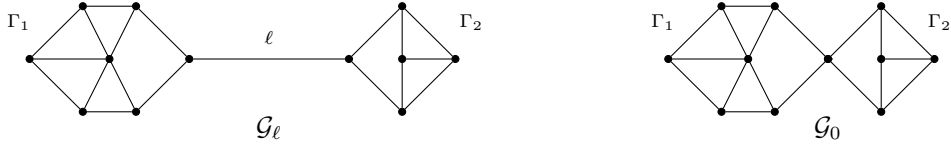

So let  $\G_\ell$ be a graph with bridging edge $e$ of length $\ell$ and let $|\G_\ell |$ be the total length of $\G_\ell$. Let also $\Gamma_1$ and $\Gamma_2$ be the two disjoint connected components in which $\G_\ell$ breaks up after the removal of $e$, i.e., $\G_\ell\setminus\{e\}=\Gamma_1\cup\Gamma_2$, and note that they do not depend on $\ell$ (see Fig. \ref{gl}). We identify $e$ with the interval $[0,\ell]$ in such a way that $x=0$ corresponds to the vertex in $\Gamma_1$, whereas $x=\ell$ to the one in $\Gamma_2$. 
\medskip

\noindent\textit{Proof of Theorem \ref{THM 4}.} We split the proof into two parts.
\medskip

\noindent\textit{First part: asymptotics for $\ell\to \infty$.} Assume without loss of generality that $|\Gamma_1|\geq |\Gamma_2|$. and set $a=|\Gamma_1|-|\Gamma_2|$.

For $\ell>a$ we define $u \in H^1(\G_\ell)$ as
\[
u_\ell(x)=\begin{cases}
1 &  \text{if }x\in\Gamma_1\\
\cos\left(\frac{\pi x}{\ell-a}\right) & \text{if } x \in [0,\ell-a] \\
-1 & \text{elsewhere on } \G_\ell.
\end{cases}
\]
By construction $u_\ell$ has zero mean value and therefore
\[
\lambda_2(\G_\ell) \le \frac{\int_{\G_\ell} |u_\ell'|^2\dx}{\int_{\G_\ell} |u_\ell|^2\dx}
 \le \frac{\int_0^{\ell-a} |u_\ell'|^2\dx}{\int_0^{\ell-a} |u_\ell|^2\dx} = \frac{\pi^2}{(\ell-a)^2}.
\]
Plugging this inequality into \eqref{mu2} with $p=6$ yields
\[
\mu_1(\G_\ell,6)\leq\frac{\pi}{2}\frac{|\G_\ell |}{\ell-a} = \frac{\pi}{2}\frac{|\Gamma_1| + |\Gamma_2|+ \ell}{\ell-a}\,,
\]
so that, passing to the limit $\ell\to \infty$  and coupling with Theorem \ref{THM 3}, we obtain
\[
\lim_{\ell\to+\infty}\mu_1(\G_\ell,6)=\frac{\pi}{2}\,.
\]
\medskip

\noindent\textit{Second part: asymptotics for $\ell\to0$.} Let $\varphi_\ell$ be an  eigenfunction associated with $\lambda_2(\G_\ell)$, normalized by $\|\varphi_\ell\|_{L^2(\G_\ell)}=1$. Let $v \in H^1(\G_\ell)$ be a fixed function supported in $\Gamma_1$, with zero mean value, and normalized by $\|v\|_{L^2(\G_\ell)}=1$. Then
\[
\int_{\G} |\varphi_\ell' |^2\dx = \lambda_2(\G_\ell) \le \int_{\Gamma_1} |v' |^2\dx
\]
for every $\ell$, from which we see that $\varphi_\ell$ is bounded in $H^1(\G_\ell)$ uniformly in $\ell$.

The restriction of $\varphi_\ell$ to $\Gamma_1\cup\Gamma_2$ is a fortiori bounded in $H^1(\Gamma_1\cup\Gamma_2)$, and therefore we may assume that, up to subsequences, it converges to a function $\psi$, weakly in $H^1(\Gamma_1\cup\Gamma_2)$ and strongly in $L^\infty(\Gamma_1\cup\Gamma_2)$. 

Since
\[
|\varphi_\ell(\ell)-\varphi_\ell(0)|  \le \int_{\G_\ell} |\varphi_\ell' |\dx \le \sqrt\ell\left(\int_{\G_\ell} |\varphi_\ell' |^2\dx\right)^{1/2} = \sqrt{\ell \lambda_2(\G_\ell)} = o(1)
\]
as $\ell \to 0$, we see that the limit function $\psi$ is well defined in the vertex where $\Gamma_1$ and $\Gamma_2$ meet when $\ell=0$, and hence it  can be considered as a function on $\G_0$. 

Furthermore,
\[
\int_{\G_0} |\psi |^2\dx = \lim_{\ell \to 0} \int_{\Gamma_1\cup\Gamma_2} |\varphi_\ell |^2\dx = 
\lim_{\ell \to 0}\left( \int_{\G_\ell} |\varphi_\ell |^2\dx -  \int_0^\ell |\varphi_\ell |^2\dx  \right) = 1
\]
and, similarly,
\[
\int_{\G_0} \psi \dx = \lim_{\ell \to 0} \int_{\Gamma_1\cup\Gamma_2} \varphi_\ell \dx = 
\lim_{\ell \to 0}\left( \int_{\G_\ell} \varphi_\ell \dx -  \int_0^\ell \varphi_\ell \dx  \right) = 0
\]
since $\|\varphi_\ell\|_{L^\infty({\G_\ell} )}$ is  bounded independently of $\ell$. Thus $\psi \in H^1(\G_0)$ has zero mean value and is normalized in $L^2$. Therefore, since $\G_0$ can be covered by cycles, we have by \eqref{EQ-eigen geq circle},
\[
\frac{4\pi^2}{|\G_0|^2} \le \int_{\G_0} |\psi' |^2\dx \le \liminf_{\ell\to 0} \int_{\Gamma_1\cup\Gamma_2} |\varphi_\ell' |^2\dx
\le  \liminf_{\ell\to 0} \int_{\G_\ell} |\varphi_\ell'|^2\dx =  \liminf_{\ell\to 0} \lambda_2(\G_\ell)
\]
which, via \eqref{mu2} with $p=6$, yields 
\[
\pi\leq\liminf_{\ell\to0}\mu_1(\G_\ell,6)\,.
\]
\hspace{\stretch{1}} $\Box$ 
\medskip

\noindent\textit{Proof of Corollary \ref{COR 2}.}  
As recalled in Section \ref{sec:prelim}, since $\G_\ell$ has no terminal edge, ground states of mass $\mu$ exist if and only if $\mu\in(0,\mu_\rr]$. On the other hand, since in addition $\G_\ell$ cannot be covered by cycles, by Theorem \ref{THM 4},
\[
\lim_{\ell \to \infty} \mu_1(\G_\ell,6) = \frac\pi 2 < \frac{\pi\sqrt 3}2 = \mur,
\]
so that $\mu_1(\G_\ell,6) < \mur$ for every $\ell$ large enough. In this range of lengths,
by Theorem \ref{THM 1}, for every $\mu \in (\mu_1(\G_\ell,6),\mur]$ the function $\cost$ is not a local minimum of the energy, and a fortiori not a ground state. Hence the ground states must be nonconstant.
\hspace{\stretch{1}} $\Box$ 
\medskip

\noindent\textit{Proof of Corollary \ref{COR 3}.} Since by Theorem \ref{THM 4}
\[
\liminf_{\ell \to 0} \mu_1(\G_\ell,6) \ge \pi > \frac{\pi\sqrt 3}2 = \mur,
\]
for every $\ell$ small enough $\mu_1(\G_\ell,6) > \mur$. Therefore in this range of lengths, by Theorem \ref{THM 1},  for every $\mu \in (\mu_1(\G_\ell,6),\mur]$ the function $\cost$ is   orbitally stable.
\hspace{\stretch{1}} $\Box$ 

\begin{rem}
\label{REM-constant ground states}
We remark that working with graphs without terminal edges and no cycle covering is up to now essential in order to find nonconstant ground states . Indeed, if $\G$ has a terminal edge, ground states exist if and only if $\mu\in(0,\mu_{\rr^+}]$, while the constant function $\cost$ locally minimizes the energy for masses up to $\mu_1\geq\pi/2>\mup$, by Theorem \ref{THM 3}. Thus $\cost$ remains a local minimizer for all the masses allowing for the existence of global minimizers. The same situation occurs when $\G$ admits a cycle covering, as ground states exist for  $\mu \le \mu_\rr=\sqrt{3}\pi/2$, whereas $\cost$ is a local minimizer of $E$ for every $\mu\leq\mu_1$ and $\mu_1\geq\pi$ by Theorem \ref{THM 3}. We believe that in all these cases the constant function $\cost$ is the ground state of the energy for every mass for which global minimizers exist, even though we cannot provide any proof of this conjecture.
\end{rem}

\end{document}